\documentclass{amsart}
\usepackage[mathscr]{euscript}
\usepackage{amsthm}
\usepackage{amssymb}

\begin{document}

\title{Another inequality inspired by Erd\H{o}s}

\author{Barbora Bat\'\i kov\'a, Tom\'a\v s J.\  Kepka and Petr C.\ N\v emec}

\address{Barbora Bat\'\i kov\'a, Department of Mathematics, CULS,
Kam\'yck\'a 129, 165 21 Praha 6 - Suchdol, Czech Republic}
\email{batikova@tf.czu.cz}

\address{Tom\'a\v s J.\ Kepka, Faculty of Education, Charles University, M.\ Rettigov\'e 4, 116 39 Praha 1, Czech Republic}
\email{kepka@karlin.mff.cuni.cz}

\address{Petr C.\ N\v emec, Department of Mathematics, CULS,
Kam\'yck\'a 129, 165 21 Praha 6 - Suchdol, Czech Republic}
\email{nemec@tf.czu.cz}



\begin{abstract} One less common inequality in positive integers is carefully examined.
\end{abstract}

\maketitle

\newtheorem{obs}{Observation}[section]
\newtheorem{remark}[obs]{Remark}
\newtheorem{lemma}[obs]{Lemma}
\newtheorem{theorem}[obs]{Theorem}
\newtheorem{proposition}[obs]{Proposition}
\newtheorem{ex}[obs]{Exercise}
\newtheorem{const}[obs]{Construction}

In 1845, J.\ Bertrand (\cite{B}, pg.\ 129) conjectured such a statement: "For every integer $n$ greater than 6 there always exists at least one prime number contained between $n-2$ and $\frac n2$." Bertrand himself noticed that his proposition (postulate) was true for all $n$ such that $7\le n<6\cdot 10^6$ and believed it remained true for all $n\ge 7$. Seven years later, P.\ L.\ Tchebychef published a complete proof of the conjecture (\cite{T}) and various simplifications of this proof appeared through passing years. It was P.\ Erd\"os in 1932 (\cite{E}) who came up with an elegant and simple proof based on different ideas -- prime decompositions of binomial coefficients (a~thorough exposition of all this can be found in \cite{C}). The approach of Erd\"os may be considered the standard one nowadays. Two remarks could turn out to be of some interest, however. Firstly, a formally weaker form of the postulate is shown in \cite{E}. Namely, that for every positive integer $n$ there is at least one prime $p$ with $n<p\le 2n$. Of course, all odd primes $p$ are of the form $p=2n-1$, $n\ge 2$, whence it is not immediately visible why the lenient form should imply the more severe one. On the other hand, the original proof by Erd\"os can easily be completed to do the work. Secondly, the proof can be definitely viewed as elementary, but no way calculus-free. On the other hand, if some assertion is formulated within elementary arithmetic, it would be nice to have a really elementary proof. Nonetheless, we are convinced that the methods developed by Erd\"os can be transmogrified to yield demonstration(s) of the Bertrand postulate if not inside then at least close to Elementary Function Arithmetic -- one of the weakest fragments of arithmetic. 
In our effort to find such a proof we came across several unusual inequalities in positive integers. One of them (the inequality (1) below) was carefully examined in \cite{BKN}. Another kindred inequality (the inequality (2)) is treated in the present short note. Purely from methodological and didactical reasons, we include also a solution of these two inequalities using elementary calculus.

\section{Preliminaries}

In the sequel, $n$ stands for a positive integer. We denote by
\begin{enumerate}
\item[(a)] $z(n)$ the largest integer such that $3z(n)<2n$;
\item[(b)] $m(n)$ the largest integer such that $m(n)^2\le 2n$;
\item[(c)] $r(n)$ the least non-negative integer such that $n\le 2^{r(n)}$;
\item[(d)] $x(n)=z(n)-(r(n)+1)m(n)$;
\item[(e)] $y(n)=2^{2n-2z(n)-m(n)+2}-\cdot n^{m(n)-1}$;
\item[(f)] $c(n)=2n-2z(n)+2$;
\item[(g)] $a(n)=2^{c(n)-m(n)}$;
\item[(h)] $b(n)=n^{m(n)-1}$.
\end{enumerate}

In \cite{BKN}, we examined the inequality
\begin{equation}
z(n)-(r(n)+1)m(n)<0\,.
\end{equation}
i.e., the inequality $x(n)<0$) and we proved the following result (and a bit more):

\begin{theorem}\label{T1} Let $n$ be a positive integer. Then $x(n)=0$ if and only if $n\in\{436,451,529,545,546\}$ and $x(n)<0$ if and only if either $n\le435$ or $n=450$ or $513\le n\le528$.
\end{theorem}
Now, our aim is to solve (in positive integers) the inequality
\begin{equation}
2^{2n-2z(n)-m(n)+2}-n^{m(n)-1}<0
\end{equation}
and to find all positive integers $n$ for which the equality holds. In other words, we want to determine the sign of $y(n)$. The answer is given in Theorem \ref{T2}.

\begin {ex}\label{R1} \rm It is not complicated to check the following handy table (without using any artificial means):

\medskip
$\begin{array}{c|c|c|c|c|c|c|c|c|c|c|}
n&1&2&3&4&5&6&7&8&9&10\\x(n)&-1&-3&-5&-4&-9&-9&-8&-11&-15&-14\\c(n)-m(n)&3&2&4&4&3&5&5&4&6&6\\y(n)&7&2&13&12&-17&-4&-17&-496&-665&-936\end{array}$

\bigskip 
$\begin{array}{c|c|c|c|c|c|c|}
n&11&12&13&14&15&16\\x(n)&-13&-13&-17&-16&-21&-20\\c(n)-m(n)&6&8&7&7&9&9\\y(n)&-1267&-1472&-28433&-38288&-50113&-65024\end{array}$
\end{ex}

In order to justify the definition of $y(n)$ (see (e)), we need the following:

\begin{lemma}\label{L13} $c(n)-m(n)\ge2$ for every positive integer $n$ and $c(n)-m(n)=2$ only for $n=2$.\end{lemma}

\begin{proof} With regard to \ref{R1}, we can assume that $n\ge10$. Then $2n\ge20$, $z(n)\ge6$, $2z(n)\ge12$, $2z(n)^2\ge12z(n)\ge10z(n)+12>10z(n)+8$, $(z(n)-1)^2=z(n)^2-2z(n)+1=\frac{2z(n)^2-4z(n)+2}2>\frac{10z(n)+8-4z(n)+2}2=3z(n)+5>3(z(n)+1)\ge2n$. Now, $(z(n)-1)^2>2n$ implies $z(n)-1>m(n)$ and $z(n)-m(n)-1\ge1$. On the other hand, $2n-3z(n)\ge1$ and, consequently, $c(n)-m(n)-3=2n-2z(n)-m(n)-1=(2n-3z(n))+(z(n)-m(n)-1)\ge2$. Thus $c(n)-m(n)\ge5$.\end{proof}

One observes readily that the sequences $z(n),m(n),r(n)$ are non-decreasing, the sequence $b(n)$ is increasing and $y(n)=a(n)-b(n)$. Further, $c(1)=c(2)=4$, $c(3k)=c(3k+1)=c(3k+2)=2k+4$ for every $k\ge2$, $c(n+3k)=c(n)+2k$ and the sequence $c(n)$ is non-decreasing.

\section{An arithmetically pure solution of (2)}

\begin{obs} \rm In order to facilitate the calculation of $x(n),y(n)$, it is useful to distinguish intervals where the values of $m(n)$ and $r(n)$ are constant (and thus the values of $x(n)$ are non-decreasing).\newline
(i) For every positive integer $n$ we denote by $d_1(n)$ ($d_2(n)$, resp.) the least (largest, resp.) positive integer $n_1$ ($n_2$, resp.) such that $m(n_1)=m(n)=m(n_2)$. We observe easily that $d_1(n)=\lceil \frac{m(n)^2}2\rceil$ and $d_2(n)=\lfloor\frac{m(n)^2}2\rfloor+m(n)$. Besides, $d_2(n)-d_1(n)=m(n)$ if $m(n)$ is even  and $d_2(n)-d_1(n)=m(n)-1$ if $m(n)$ is odd.\newline
(ii) For every positive integer $n$ we denote by $e_1(n)$ ($e_2(n)$, resp.) the least (largest, resp.) positive integer $n_3$ ($n_4$, resp.) such that $r(n_3)=r(n)=r(n_4)$. We observe immediately that $e_1(1)=1=e_2(1)$ and $e_1(n)=2^{r(n)-1}+1$, $e_2(n)=2^{r(n)}$ for $n\ge2$. Besides, $e_2(n)-e_1(n)=2^{r(n)-1}-1$ for $n\ge2$.\newline
(iii) For every positive integer $n$ we put $f_1(n)=\max(d_1(n),e_1(n))$ and $f_2(n)=\min(d_2(n),e_2(n))$. That is, $f_1(1)=1=f_2(1)$ and $f_1(n)=\max\lceil\frac{m(n)^2}2\rceil,2^{r(n)-1}+1)$, $f_2(n)=\min(\lfloor\frac{m(n)^2}2\rfloor+m(n),2^{r(n)})$. Of course, $f_1(n)\le n\le f_2(n)$. Clearly, $f_1(n)$ ($f_2(n)$, resp.) is the least (largest, resp.) positive integer $k$ such that $m(k)=m(n)$ and $r(k)=r(n)$.\newline
(iv) We have $f_2(2^k)=2^k$ for every $k\ge0$. If $k\ge1$ is odd then $f_1(2^k)=2^k=f_2(2^k)$.\newline
(v) Clearly, if $n_1,n_2$ are positive integers then $m(n_1)=m(n_2)$ and $r(n_1)=r(n_2)$  iff $f_1(n_1)=f_1(n_2)$ iff $f_2(n_1)=f_2(n_2)$. Further, if $n$ is 
a~positive integer and $f_1(n)\le n_1\le n_2\le f_2(n)$ then $x(n_1)\le x(n_2)=x(n_1)+z(n_2)-z(n_1)$. \newline
\end{obs}

\begin{const}\label{R4}
\rm  Put $g(1)=1$ ($=f_2(1)$) and $g(i+1)=f_2(g(i))+1$ for each $i\ge1$. Clearly, $g(i)\le n\le g(i+1)-1$, $i=1,2,\dots$, are maximal intervals where $r(n),m(n)$ are constant and if $g(i)\le n\le g(i+1)-1$ then $x(n)=z(n)-(r(g(i))+1)m(g(i))$.

The first 41 intervals $g(i)\le n\le g(i+1)-1$ are presented in the following list (the necessary calculations, although a bit fatigueant, are extremely easy):
\begin{enumerate}
\item If $n=1$ then $r(n)=0$, $m(n)=1$, $x(n)=-1$.
\item If $n=2$ then $r(n)=1$, $m(n)=2$, $x(n)=-3$.
\item If $3\le n\le4$ then $r(n)=2$, $m(n)=2$, $-5\le x(n)\le-4$.
\item If $5\le n\le7$ then $r(n)=3$, $m(n)=3$, $-9\le x(n)\le-8$.
\item If $n=8$ then $r(n)=3$, $m(n)=4$, $x(n)=-11$.
\item If $9\le n\le12$ then $r(n)=4$, $m(n)=4$, $-15\le x(n)\le-13$.
\item If $13\le n\le16$ then $r(n)=4$, $m(n)=5$, $-17\le x(n)\le-15$.
\item If $n=17$ then $r(n)=5$, $m(n)=5$, $x(n)=-19$.
\item If $18\le n\le24$ then $r(n)=5$, $m(n)=6$, $-25\le x(n)\le-21$.
\item If $25\le n\le31$ then $r(n)=5$, $m(n)=7$, $-26\le x(n)\le-22$.
\item If $n=32$ then $r(n)=5$, $m(n)=8$, $x(n)=-27$.
\item If $33\le n\le40$ then $r(n)=6$, $m(n)=8$, $-35\le x(n)\le-30$.
\item If $41\le n\le49$ then $r(n)=6$, $m(n)=9$, $-36\le x(n)\le-31$
\item If $50\le n\le60$ then $r(n)=6$, $m(n)=10$, $-37\le x(n)\le-31$.
\item If $61\le n\le64$ then $r(n)=6$, $m(n)=11$, $-37\le x(n)\le-35$.
\item If $65\le n\le71$ then $r(n)=7$, $m(n)=11$, $-45\le x(n)\le-41$.
\item If $72\le n\le84$ then $r(n)=7$, $m(n)=12$, $-49\le x(n)\le-41$.
\item If $85\le n\le97$ then $r(n)=7$, $m(n)=13$, $-48\le x(n)\le-40$.
\item If $98\le n\le112$ then $r(n)=7$, $m(n)=14$, $-47\le x(n)\le-38$.
\item If $113\le n\le127$ then $r(n)=7$, $m(n)=15$, $-45\le x(n)\le-36$.
\item If $n=128$ then $r(n)=7$, $m(n)=16$, $x(n)=-43$.
\item If $129\le n\le144$ then $r(n)=8$, $m(n)=16$, $-59\le x(n)\le-49$
\item If $145\le n\le161$ then $r(n)=8$, $m(n)=17$, $-57\le x(n)\le-46$.
\item If $162\le n\le180$ then $r(n)=8$, $m(n)=18$, $-55\le x(n)\le-43$.
\item If $181\le n\le199$ then $r(n)=8$, $m(n)=19$, $-51\le x(n)\le-39$.
\item If $200\le n\le220$ then $r(n)=8$, $m(n)=20$, $-47\le x(n)\le-34$.
\item If $221\le n\le241$ then $r(n)=8$, $m(n)=21$, $-42\le x(n)\le-29$.
\item If $242\le n\le256$ then $r(n)=8$, $m(n)=22$, $-37\le x(n)\le-28$.
\item If $257\le n\le264$ then $r(n)=9$, $m(n)=22$, $-49\le x(n)\le-45$.
\item If $265\le n\le287$ then $r(n)=9$, $m(n)=23$, $-54\le x(n)\le-39$.
\item If $288\le n\le312$ then $r(n)=9$, $m(n)=24$, $-49\le x(n)\le-33$.
\item If $313\le n\le337$ then $r(n)=9$, $m(n)=25$, $-42\le x(n)\le-26$.
\item If $338\le n\le364$ then $r(n)=9$, $m(n)=26$, $-35\le x(n)\le-18$.
\item If $365\le n\le391$ then $r(n)=9$, $m(n)=27$, $-27\le x(n)\le-10$.
\item If $392\le n\le420$ then $r(n)=9$, $m(n)=28$, $-19\le x(n)\le-1$.
\item If $421\le n\le449$ then $r(n)=9$, $m(n)=29$, $-10\le x(n)\le9$.
\item If $450\le n\le480$ then $r(n)=9$, $m(n)=30$, $-1\le x(n)\le19$.
\item If $481\le n\le511$ then $r(n)=9$, $m(n)=31$, $10\le x(n)\le30$.
\item If $n=512$ then $r(n)=9$, $m(n)=32$, $x(512)=21$.
\item If $513\le n\le544$ then $r(n)=10$, $m(n)=32$, $-11\le x(n)\le10$.
\item If $545\le n\le577$ then $r(n)=10$, $m(n)=32$, $0\le x(n)\le21$.
\end{enumerate}
\end{const}

\begin{lemma}\label{L4} Let $u,v,n$ be positive integers such that $u\le n\le v$ and $m(u)=m(v)=t$.\newline
{\rm(i)} $a(u)-b(v)\le y(n)=a(n)-b(n)=2^{c(n)-t}-n^{t-1}\le a(v)-b(u)$.\newline
{\rm(ii)} If $b(u)=u^{t-1}>2^{c(v)-t}=a(v)$ then $y(n)<0$.\newline
{\rm(iii)} If $b(v)=v^{t-1}<2^{c(u)-t}=a(u)$ then $y(n)>0$.\newline
{\rm(iv)} If $n+1\le v$ and $c(n)=c(n+1)$ then $y(n+1)<y(n)$.\end{lemma}

\begin{proof} It follows from the fact that the sequence $c(n)$ is non-decreasing and the sequence $b(n)$ is increasing.\end{proof} 

\begin{lemma}\label{L5} Let $n$ be a positive integer. Put $r(n)=s$ and $m(n)=t$. Then:\newline
{\rm(i)} If $c(n)\le s(t-1)+1$ then $y(n)<0$.\newline
{\rm(ii)} If $c(n)>s(t-1)+t$ then $y(n)>0$.
\end{lemma}

\begin{proof} Since $r(n)=s$, we have $2^{s-1}<n\le 2^s$. If $c(n)\le st-s+1$ then $c(n)-t\le(s-1)(t-1)$, hence $2^{c(n)-m(n)}\le\left(2^{s-1}\right)^{t-1}<n^{t-1}$ and $y(n)<0$. If $c(n)>st-s+t$ then $s(t-1)<c(n)-t$, and therefore $0\le(2^s)^{t-1}-n^{t-1}<2^{c(n)-t}-n^{t-1}=y(n)$.\end{proof}

\begin{lemma}\label{L2}
If $n\ge1$ is such that $y(n)\le0$ then $x(n)\le-r(n)-3\le-6$.
\end{lemma}

\begin{proof}
In view of  \ref{R1}, we have $n\ge 5$. As $3z(n)\le2n-1$, we have $z(n)-m(n)+3\le z(n)-m(n)+2+2n-3z(n)=c(n)-m(n)$. However $y(n)\le0$, hence $2^{c(n)-m(n)}\le n^{m(n)-1}\le2^{r(n)(m(n)-1)}$ and $z(n)-m(n)+3\le c(n)-m(n)\le r(n)m(n)-r(n)$. Then $x(n)=z(n)-r(n)m(n)-m(n)\le-r(n)-3\le-6$.\end{proof}

\begin{proposition}\label{P1}
$y(n)>0$ for every $n\ge404$.
\end{proposition}

\begin{proof}
Assume $n\ge404$. Then $r(n)\ge9$, $-r(n)-3\le-12<-11$. On the other hand, taking into account \ref{T1}, \ref{R4} and the fact that $x(404)=-11$, we have $-11\le x(n)$. Thus $-r(n)-3<x(n)$ and $y(n)>0$ by \ref{L2}.
\end{proof}

\begin{theorem}\label{T2} {\rm(i)} $y(n)<0$ if and only if either $5\le n\le 335$ or $338\le n\le 350$ or $365\le n\le368$.\newline
{\rm(ii)} $y(n)\ne0$ for every positive integer $n$.\newline
{\rm(iii)} $y(n)>0$ if and only if either $1\le n\le4$ or $336\le n\le337$ or $351\le n\le364$ or $n\ge369$.\end{theorem}

\begin{proof} With respect to \ref{R1} and \ref{P1}, we can restrict ourselves to $17\le n\le403$. We will make use of \ref{R4} and compare the powers of 2 with those of $n$ (the necessary calculations are non-demanding in principle, nonetheless they might turn out to be somewhat tiresome occasionally -- an abacus and/or a homemade table of powers of 2 will help a lot). The proof will be divided into several parts.\newline
(i) From \ref{L5}(i) and the table in \ref{R4}, we see that $y(n)<0$ for $n\le220$ (i.e., for $i=1,\dots,26$).\newline
(ii) Let $221\le n\le241$. Then $r(n)=8$, $m(n)=21$ and $c(241)=164$. As $221^3>2^{23}$, we have $221^{20}=(221^3)^6\cdot221^2>221^{138}\cdot221^2>2^{153}$ and $y(n)<0$ for $221\le n\le241$ by \ref{L4}(ii).\newline
(iii) Let $242\le n\le 264$. Then $m(n)=22$, $c(264)=180$ and $242^4>2^{31}$. Thus $242^{21}=(242^4)^5\cdot242>2^{31\cdot5}\cdot242>2^{158}$ and $y(n)<0$ by \ref{L4}(ii).\newline
(iv) Let $265\le n\le 287$. Then $m(n)=23$, $c(287)=194$, $265^{22}>256^{22}=(2^8)^{22}=2^{176}>2^{194-23}$ and $y(n)<0$ by \ref{L4}(ii).\newline
(v) Let $288\le n\le312$. Then $r(n)=9$, $m(n)=24$ and $c(312)=212$. As $c(308)=c(306)=208\le 9\cdot23+1$, $y(n)<0$ for $288\le n\le308$ by \ref{L5}(i). Further, $309>2^8$ and $309^4>2^{33}$. Hence $309^{23}=(309^4)^5\cdot309^3>2{33\cdot5}\cdot2^{8\cdot3}>2^{188}$ and $y(n)<0$ for $309\le n\le312$ by \ref{L4}(ii).\newline
(vi) Let $313\le n\le337$. Then $r(n)=9$, $m(n)=25$ and $c(337)=228$. As $c(320)=216\le9\cdot24+1$, $y(n)<0$ for $313\le n\le320$ by \ref{L5}(i).
Further, $c(329)=222$ and (see (v)) $321^4>2^{33}$. Thus $320^{24}>2^{33\cdot6}$ and $y(n)<0$ for $320\le n\le329$ by \ref{L4}(ii). Moreover, $c(332)=224$ and $330^3>2^{25}$. Hence $330^{24}>2^{200}$ and $y(n)<0$ for $330\le n\le332$ by \ref{L4}(ii). Further, $333^8>1.5\cdot10^{20}>1.48\cdot10^{20}>2^{67}$ (we have $2^{67}=147,573,952,589,676,412,928$), and hence $333^{24}>2^{67\cdot3}=2^{201}$. As $c(333)=226$, $y(333)<0$ and $y(n)<0$ for $333\le n\le335$ by \ref{L4}(iv). Finally,  we have $m(337)=25$, $c(337)=c(336)=228$, $337^2<2^{17}$ and $337^5<2^{42}$. Thus $337^{24}<2^{42\cdot4}\cdot2^{17\cdot2}=2^{202}$ and $y(337)>0$. Then $y(336)>0$ by \ref{L4}(iv).\newline
(vii) Let $338\le n\le364$. Then $r(n)=9$ and $m(n)=26$. As $c(350)=236$ and $338^5>2^{42}$, we have $338^{25}>2^{42\cdot5}$ and $y(n)>0$ for $338\le n\le350$ by \ref{L4}(ii). As $353^2<2^{17}$ and $353^{15}<1.65\cdot10^{38}<1.7\cdot10^{38}<2^{127}$ (we have $2^{127}= 170,141,183,460,469,231,731,687,303,715,884,105,728$), $353^{25}=353^{15}\cdot(353^2)^5>2^{127}\cdot2^{17\cdot5}=2^{212}$, hence $y(353)>0$ and $y(n)>0$ for $351\le n\le353$ by \ref{L4}(iv). Further, $356^2<2^{17}$, hence $356^{25}<2^{17\cdot12}\cdot356<2^{213}$. As $c(354)=240$, $y(n)>0$ for $354\le n\le356$  by \ref{L4}(iii). However, $364^5<2^{43}$, and hence $364^{25}<2^{215}$. As $c(357)=242$, $y(n)>0$ for $357\le n\le364$ by \ref{L4}(iii). \newline
(viii) Let $365\le n\le 391$. Then $r(n)=9$ and $m(n)=27$. We have $c(368)=248$ and $365^2>2^{17}$. Thus $365^{26}>2^{17\cdot13}=2^{221}$ and $y(n)<0$ for $365\le n\le368$ by \ref{L4}(ii). Further, as $371^3<2^{26}$, $371^5<2^{43}$ and $371^9<1.4\cdot10^{23}<1.5\cdot10^{23}<2^{77}=151,115,727,451,828,646,838,272$, we have $371^{26}=(371^9)^2\cdot371^5\cdot371^3<2^{77\cdot2+43+26}=2^{223}$, hence $y(371)>0$ and $y(n)>0$ for $369\le n\le371$ by \ref{L4}(iv). Further, $386^5<2^{43}$, and hence $386^{26}<2^{43\cdot5}\cdot386<2^{224}$. As $c(372)=242$, we have $y(n)>0$ for $372\le n\le 386$ by \ref{L4}(iii).
Finally, we have $c(387)=262>9\cdot26+27$, $y(n)>0$ for $387\le n\le391$ by \ref{L5}(ii). \newline
(ix) Let $392\le n\le403$. Then $r(n)=9$ and $m(n)=28$. As $c(402)=c(403)=c(404)$ and $y(404)>0$, $y(n)>0$ for $n=402,403$ by \ref{L4}(iv). As $c(392)=264$ and $402^3<2^{26}$, we have $402^{27}<2^{26\cdot9}<2^{264-28}$ and $y(n)>0$ for $392\le n\le402$ by \ref{L4}(iii). \end{proof}

\section{A solution of (1) and (2) using elementary calculus}

\begin{remark}\label{31}
\rm Define a function
$$F(x)=\frac23x+A-B\sqrt{2x}+C\log_2x-\sqrt{2x}\log_2x\,,$$
where $A,B,C\in\mathbb R$.

It is an easy exercise to prove the following assertions. 
\begin{enumerate}
\item $\lim_{x\to\infty}F(x)=\infty$.
\item $F'(x)=\frac23-\frac B{\sqrt{2x}}+\frac C{x\ln2}-\frac{\ln x}{\sqrt{2x}\ln2}-\frac{\sqrt{2x}}{x\ln2}$, $\lim_{x\to\infty}F'(x)=\frac23$.
\item $F''(x)=\frac{\sqrt{x}\ln(2^Bx)-2\sqrt2C}{2\sqrt2x^2\ln2}$.
\item Functions $F,F',F''$ are continuous on $(0,\infty)$.
\item $F''(x)>0$ for $x\in(d,\infty)$, where $d\ge2$ and $d\ge\frac{{\rm e}^{2C}}{2^B}$. Thus $F$ is convex and $F'$ is ascending on $\langle d,\infty)$, and hence there is at most one $c\in(d,\infty)$ with $F'(c)=0$.
\item Either $F$ is ascending on $\langle d,\infty)$ or there is exactly one $c\in(d,\infty)$ such that $F$ is descending on $\langle d,c\rangle$ and ascending on $\langle c,\infty)$ (see (2) and (5)).
\item If $F(x_0)<0$ for some $x_0\in\langle d,\infty)$ then there is exactly one $k\in(d,\infty)$ such that $F(x)<0$ for $x\in\langle x_0,k)$ and $F(x)>0$ for $x\in(k,\infty)$ (see (6) and (1)).
\end{enumerate}
\end{remark}

\begin{lemma}\label{32}
Let $n$ be a positive integer.\newline
{\rm(i)} If the inequality $(1)$  holds then $n\le560$.\newline
{\rm(ii)} If $n\ge561$ then $x(n)>0$.
\end{lemma}

\begin{proof}
In \ref{31}, put $A=-1$, $B=2$ and $C=0$. Thus we obtain a function 
$$f(x)=\frac23x-1-\sqrt{2x}\log_2x-2\sqrt{2x}\,.$$
 We have $f(1)=\frac23-1-2\sqrt2<0$ and $f''(x)-\frac{\ln4x}{2\sqrt2x\sqrt x\ln2}>0$ for $x\in(1,\infty)$ (see \ref{31}(3) for $A=-1$, $B=2$, $C=0$). From \ref{31}(7) follows that there is exactly one $k\in(1,\infty)$ such that $f(x)<0$ for $x\in\langle1,k\rangle$ and $f(x)>0$ for $x\in(k,\infty)$. As $f(560)<0$ and $f(561)>0$, we have $k\in(560,561)$. 

If $n$ is a positive integer then $z(n)\ge \frac23n-1$, $\sqrt{2n}\ge m(n)$ and $\log_2n+1>r(n)$, and hence $f(n)=\frac23n-1-(\log_2n+2)\sqrt{2n}<z(n)-(r(n)+1)m(n)=x(n)$. If the inequality (1) holds then $f(n)<0$, and hence $n\le560$. On the other hand, if $n\ge 561$ then $0<f(n)<x(n)$.
\end{proof}

\begin{lemma}\label{33}
If $n$ is a positive integer such that $n\le384$ then the inequality {\rm(1)} holds.
\end{lemma}

\begin{proof}
In \ref{31}, put $A=B=C=1$. Thus we have a function
$$g(x)=\frac23x-(\log_2x+1)(\sqrt{2x}-1)=\frac23x-\sqrt{2x}\log_2x+\log_2x-\sqrt{2x}+1\,.$$
Then $g(4)=\frac83+3-3\sqrt8<0$ and $g''(x)=\frac{\sqrt x\ln2x-2\sqrt2}{2\sqrt2x^2\ln2}>0$ for $x\in(4,\infty)$ (see \ref{31}(3) for $A=B=C=1$). From \ref{31}(7) follows that there is exactly one $k\in(4,\infty)$ such that $g(x)<0$ for $x\in\langle4,k)$ and $g(x)>0$ for $x\in(k,\infty)$. However $g(384)<0$ and $g(385)>0$, and hence $k\in(384,385)$.

If $n$ is a positive integer then $z(n)<\frac23n$, $m(n)>\sqrt{2n}-1$ and $r(n)+1\ge\log_2n+1$. Thus $f(n)=\frac23n-(\log_2n+1)(\sqrt{2n}-1)>z(n)-(r(n)+1)m(n)=x(n)$. If $x(n)\ge0$ then $n\ge4$ (see \ref{R1}) and $g(n)>0$. Hence $n\ge 385$.
\end{proof}

\begin{remark} \rm With respect to \ref{32} and \ref{33}, it remains to examine $x(n)$ for $385\le n\le560$. We must investigate intervals where $x(n)$ is non-decreasing, i.e., intervals $g(i)\le n\le g(i+1)-1$. It is clear from \ref{R4} that we must investigate values $x(n)$ for $34\le i\le41$. If $i=34,35$ then $x(n)<0$ and if $i=38,39$ then $x(n)>0$.\newline
(i) For $i=36$ we have $x(n)=0$ iff $z(n)=290$ iff $n=436$.\newline
(ii) For $i=37$ (then $l_i=480$) we have $x(n)=0$ iff $z(n)=300$ iff $n=451$.\newline
(iii) For $i=40$ we have $x(n)=0$ iff $z(n)=352$ iff $n=529$.\newline
(vi) Finally, for $i=41$ we have $x(n)=0$ iff $z(n)=363$. As this is an odd number, we obtain $n=545$ or $n=546$.

 Now, the proof of \ref{T1} is finished.\qed\end{remark}

\begin{remark}\label{3R1} \rm For a positive integer $n$ put\begin{enumerate}
\item $Y(n)=2n-2z(n)+2-m(n)+(1-m(n))\log_2n$;
\item $d(n)=-m(n)+(1-m(n))\log_2n$.
\end{enumerate}
Clearly, the sequence $d(n)$ is decreasing, $Y(n)=c(n)+d(n)$ and the inequality (2) holds if and only if $Y(n)<0$. Further, it is easy to see that $Y(n)>0$ for $n\le4$ and $Y(n)<0$ for $5\le n\le10$.\end{remark}

\begin{lemma}\label{34} If $Y(n)<0$ then $5\le n\le379$.\end{lemma}

\begin{proof} In \ref{31} put $A=2$, $B=C=1$. We obtain a function
$$h(x)=\frac23x+2-\sqrt{2n}+\log_2x-\sqrt{2n}\log_2x\ .$$ 
Using \ref{31}(3),(7) for $A=2$, $B=1$, $C=1$, we see that $h(4)=\frac83+4-3\sqrt8<0$, $h''(x)=\frac{\sqrt x\ln2-2\sqrt2}{2\sqrt2x^2\ln2}$, $h''(x)>0$ for $x\in(4,\infty)$ and there is exactly one $k\in(4,\infty)$ such that $h(x)<0$ for $x\in\langle4,k)$ and $h(x)>0$ for $x\in(k,\infty)$. It is easy to verify that $h(379)<0$ and $f(380)>0$, and hence $k\in(379,389)$. 

As $3z(n)<2n$ and $m(n)^2\le2n$, $Y(n)>\frac23n+2-\sqrt{2n}+(1-\sqrt{2n})\log_2n=h(n)$ for every positive integer $n$. However $h(n)<Y(n)<0$, and hence $n\le379$. Finally, $Y(n)>0$ for $n\le4$.
\end{proof}

\begin{lemma}\label{35} If $Y(n)\ge0$ then either $1\le n\le 4$ or $n\ge 325$.\end{lemma}

\begin{proof} Suppose, on the contrary, that $Y(n)\ge0$. In \ref{31} put $A=5$, $B=1$, $C=2$. we obtain a function a function 
$$k(x)=\frac23x+5-\sqrt{2x}+2\log_2x-\sqrt{2x}\log_2x\ .$$
By \ref{31}(3),(7) for $A=5$, $B=1$ and $C=2$, we have $k(9)=11-\sqrt{18}+2\log_29-\sqrt{18}\log_29<0$, $k''(x)=\frac{\sqrt x\ln2x-4\sqrt2}{2\sqrt2x^2\ln2}$, $k''(x)>0$ for $x\in(9,\infty)$ and there is exactly one $k\in(9,\infty)$ such that $k(x)<0$ for $x\in\langle9,k)$ and $k(x)>0$ for $x\in(k,\infty)$. However $f(324)<0$, $f(325)>0$, and hence $k\in(324,325)$.
As $3(z(n)+1)\ge2n$ and $(m(n)+1)^2>2n$, we have $z(n)\ge\frac23n-1$, $m(n)>\sqrt{2n}-1$ and $Y(n)=2(n-z(n))+2-m(n)+(1-m(n))\log_2n<\frac23n+5-\sqrt{2n}+2\log_2n-\sqrt{2n}\log_2n=k(n)$ for every positive integer $n$. Taking into account \ref{3R1}, we have $k(n)>Y(n)\ge0$, and hence either $n\le4$ or $n\ge325$. 
\end{proof}

\begin{lemma} $Y(n)\ne0$ for every positive integer $n$.\end{lemma}

\begin{proof} Suppose, on the contrary, that $n$ is such that $Y(n)=0$. Taking into account \ref{31} and \ref{35}, we have $n\ge325$ and $k(n)<0$, hence $n\le379$. As $Y(n)=2(n-z(n))+2-m(n)+(1-m(n))\log_2n=0$, $q=\log_2n=\frac{2(n-z(n))+2-m(n)}{m(n)-1}$ is a~rational number. However $n=2^q$ is an integer, hence $q$ is also an integer. As $2^8<325\le n\le379<2^9$, we have a contradiction.
\end{proof}

\begin{remark}\rm If $5\le n\le324$ then $Y(n)<0$ by \ref{35}. With respect to \ref{34}, it remains to determine the sign of $Y(n)$ for $325\le n\le379$. 
We will distinguish several cases.\newline
(i) If $325\le n\le 332$ then $c(n)\le c(325)+4$, and hence $Y(n)<Y(325)+4$. However $Y(325)\,(\approx -5.26)<-5$ and $y(n)<0$.\newline
(ii) If $333\le n\le335)$ then $c(n)=c(333)$, and hence $Y(n)\ge Y(337)\,(\approx1.48)>0$.\newline
(iii) If $338\le n\le350$ then $c(n)\le c(338)+8$ and $Y(n)\le Y(338)\,(\approx-8.02)+8<0$.\newline 
(iv) If $351\le n\le353$ then $c(n)=c(353)$ and $Y(n)\ge Y(353)\,(\approx0.41)>0$.\newline 
(v) If $354\le n\le 364$ then $c(n)\ge c(354)=240$ and $d(n)\ge d(364)\,(\approx-238.7)>-239$, and hence $Y(n)>1$.\newline
(vi) If $365\le n\le368$ then $c(n)\le\c(365)+2$ and $Y(n)\le Y(365)\,(\approx-2.4)+2<0$.\newline
(vii) If $369\le n\le371$ then $c(n)=c(369)$ and $Y(n)\ge Y(371)\,(\approx1.08)>0$.\newline
(viii) Finally, let $371\le n\le379$. Then $m(n)=27$ (by the table in \ref{R4}, this is true till $n=391$) and $d(n)=-27-26\log_2n$. Define $\Delta Y(n)=Y(n+3)-Y(n)$ and $\Delta d(n)=d(n+3)-d(n)$. As $c(n+3)=c(n)+2$, we have $\Delta Y(n)=2+\Delta d(n)$. If $n+3\le391$ then the sequence $\Delta d(n)=26\log_2\frac n{n+3}$ is increasing and $\Delta d(371)=26\log_2\frac{371}{374} ,(\approx-0.3)>-0.4$. Thus $\Delta Y(n)\ge2-0.4>0$. If $i$ is non-negative and $372+3i+2\le391$ then (as $(c(372+3i)=c(372+3i+1)=c(372+3i+2)$) $Y(372+3i)>Y(372+3i+2)=Y(371+3(i+1))\ge Y(371)+(i+1)\Delta Y(371)>Y(371)>0$.

Now, the proof of \ref{T2} is finished.\qed
\end{remark}

\end{document}